\documentclass[reqno,11pt]{amsart}
\numberwithin{equation}{section}

\usepackage{color,graphics,epsfig,amsmath}

\newcommand{\calI}{\mathcal{I}}

\newcommand{\calL}{\mathcal{L}}

\newcommand{\mA}{\mathbb{A}}

\newcommand{\mC}{\mathbb{C}}
\newcommand{\mD}{\mathbb{D}}

\newcommand{\mF}{\mathbb{F}}

\newcommand{\mN}{\mathbb{N}}

\newcommand{\mR}{\mathbb{R}}
\newcommand{\mS}{\mathbb{S}}
\newcommand{\mT}{\mathbb{T}}

\newcommand{\mZ}{\mathbb{Z}}

\newcommand{\inv}{{\textrm{inv }}}

\newcommand{\nm}{\,\rule[-.6ex]{.13em}{2.3ex}\,}

\newtheorem{theorem}{Theorem}[section]
\newtheorem{lemma}[theorem]{Lemma}
\newtheorem{corollary}[theorem]{Corollary}
\newtheorem{proposition}[theorem]{Proposition}

\theoremstyle{definition}

\newtheorem{remark}[theorem]{Remark}

\theoremstyle{definition}
\newtheorem{definition}[theorem]{Definition}

\theoremstyle{definition}
\newtheorem{notation}[theorem]{Notation}

\begin{document}



\title[A new $\nu$-metric when $R=H^\infty$]{Extension of the
  $\nu$-metric for stabilizable plants over $H^\infty$}

\author{Amol Sasane}
\address{Department of Mathematics, London School of Economics,
    Houghton Street, London WC2A 2AE, United Kingdom.}
\email{a.j.sasane@lse.ac.uk}

\begin{abstract}
  An abstract $\nu$-metric was introduced in \cite{BalSas}, with a
  view towards extending the classical $\nu$-metric of Vinnicombe from
  the case of rational transfer functions to more general nonrational
  transfer function classes of infinite-dimensional linear control
  systems. Here we give an important concrete special
  instance of the abstract $\nu$-metric, namely the case when the
  ring of stable transfer functions is the Hardy algebra $H^\infty$,
  by verifying that all the assumptions demanded in the abstract
  set-up are satisfied. This settles the open question implicit
  in \cite{BalSas2}.
\end{abstract}

\maketitle

\section{Introduction}

We recall the general {\em stabilization problem} in control theory.
Suppose that $R$ is a commutative integral domain with identity
(thought of as the class of stable transfer functions) and let
$\mF(R)$ denote the field of fractions of $R$. The stabilization
problem is: Given $P\in (\mF(R))^{p\times
  m}$ (an unstable plant transfer function),
find $C \in (\mF(R))^{m\times p}$ (a stabilizing controller
  transfer function), such that 
$$
H(P,C):= \left[\begin{array}{cc} P \\ I \end{array} \right]
(I-CP)^{-1} \left[\begin{array}{cc} -C & I \end{array} \right]
 \in R^{(p+m)\times (p+m)} \textrm{ (is stable).}
$$
In the {\em robust stabilization problem}, one goes a step
further.  One knows that the plant is just an approximation of
reality, and so one would really like the controller $C$ to not only
stabilize the {\em nominal} plant $P_0$, but also all sufficiently
close plants $P$ to $P_0$.  The question of what one means by
``closeness'' of plants thus arises naturally. So one needs a
function $d$ defined on pairs of stabilizable plants such that
\begin{enumerate}
\item $d$ is a metric on the set of all stabilizable plants,
\item $d$ is amenable to computation, and
\item stabilizability is a robust property of the plant with respect
  to this metric (that is, whenever a plant $P_0$ is stabilized by
  a controller $C$, then there is a small enough neighbourhood of the plant $P_0$
  consisting of plants which are stabilized by the same controller $C$).
\end{enumerate}
Such a desirable metric, was introduced by Glenn Vinnicombe in
\cite{Vin} and is called the $\nu$-{\em metric}. In that paper,
essentially $R$ was taken to be the rational functions without poles
in the closed unit disk or, more generally, the disk algebra, and the
most important results were that the $\nu$-metric is indeed a metric
on the set of stabilizable plants, and moreover, one has the
inequality that if $P_0,P\in \mS(R,p,m)$, then
$$
\mu_{P,C} \geq \mu_{P_0,C}-d_{\nu}(P_0,P),
$$
where $\mu_{P,C}$ denotes the {\em stability margin} of the pair
$(P,C)$, defined by
$$
\mu_{P,C}:=\|H(P,C)\|_\infty^{-1}.
$$
This implies in particular that stabilizability is a robust property
of the plant.

The problem of what happens when $R$ is some other ring of stable
transfer functions of infinite-dimensional systems was left open in
\cite{Vin}. This problem of extending the $\nu$-metric from the
rational case to transfer function classes of infinite-dimensional
systems was addressed in \cite{BalSas}. There the starting point in
the approach was abstract. It was assumed that $R$ is any commutative
integral domain with identity which is a subset of a Banach algebra
$S$ satisfying certain assumptions, labeled (A1)-(A4), which are
recalled in Section~\ref{section_abstract_nu_metric}.  Then an
``abstract'' $\nu$-metric was defined in this setup, and it was shown
in \cite{BalSas} that it does define a metric on the class of all
stabilizable plants.  It was also shown there that stabilizability is
a robust property of the plant.

In \cite{Vin}, it was suggested that the $\nu$-metric in the case when
$R=H^\infty$ might be defined as follows. (Here $H^\infty$ denotes the
algebra of bounded and holomorphic functions in the unit disk $\{z\in
\mC:|z|<1\}$.) Let $P_1, P_2$ be unstable plants with the normalized
left/right coprime factorizations
\begin{eqnarray*}
P_1&=& N_{1} D_{1}^{-1}= \widetilde{D}_{1}^{-1} \widetilde{N}_{1},\\
P_2&=& N_{2} D_{2}^{-1}= \widetilde{D}_{2}^{-1} \widetilde{N}_{2},
\end{eqnarray*}
where $N_1, D_1, N_2, D_2, \widetilde{N}_{1},
\widetilde{D}_{1},\widetilde{N}_{2}, \widetilde{D}_{2}$ are matrices
with $H^\infty$ entries. Then
$$
d_{\nu} (P_1,P_2 )=\left\{\begin{array}{ll}
\|\widetilde{G}_2 G_1\|_\infty & \textrm{if }
 T_{G_1^* G_2} \textrm{ is Fredholm with Fredholm index } 0,\\
  1 & \textrm{otherwise}.
\end{array}
\right.
$$
Here $G_{k},
\widetilde G_{k}$ arise from $P_{k}$ ($k = 1, 2$) according to the
notational conventions given in Subsection~\ref{subsec5} below, and $\cdot^\ast$ has the usual meaning, namely: $G_1^*(\zeta)$ is the
transpose of the matrix whose entries are complex conjugates of the
entries of the matrix $G_1(\zeta)$, for $\zeta \in \mT$. Also
in the above, for a matrix $M\in (L^\infty)^{p\times m}$, $T_M$
denotes the {\em Toeplitz operator} from $(H^2)^m$ to $(H^2)^p$, given
by
$$
T_M \varphi=\textrm{P}_{(H^2)^p}(M\varphi) \quad (\varphi \in (H^2)^m)
$$
where $M\varphi$ is considered as an element of $(L^2)^p$ and
$\textrm{P}_{(H^2)^p}$ denotes the canonical orthogonal projection from
$(L^2)^p$ onto $(H^2)^p$.

In \cite{BalSas2}, we showed that the above does work for the
case when $R$ is the smaller class $QA$ of quasianalytic functions
in the unit disk. We proved this by showing that this case is
just a special instance of the abstract $\nu$-metric introduced in
\cite{BalSas}. A perusal of the extensive literature on Fredholm
theory of Toeplitz operators from the 1970s lead to this choice of
$R=QA$ and $S=QC$ (the class of quasicontinuous functions)
as conceivably the most general subalgebras of $H^\infty$
and $L^\infty$ which fit the setup of \cite{BalSas}.

In this article, we use a different idea to tackle the problem of
defining a new metric in the case when $R=H^\infty$. We first notice
that when $R$ is the disk algebra $A(\mD)$, then there is no problem
in defining the $\nu$-metric; see \cite[\S5.1]{BalSas}. We then handle
the $H^\infty$ case by using the observation that the restrictions of a function
$f\in H^\infty$ to the smaller disks with radii $r<1$ give rise to
elements in the disk algebra by dilating these restrictions to bigger
disks of radius $1$. In other words, $f_r$ defined via
$$
f_r(z)= f(r z)\quad (z\in \mD).
$$
are all elements of $A(\mD)$. We then use these restrictions in a suitable
manner to define the $\nu$-metric.

The paper is organized as follows:
\begin{enumerate}
\item In Section~\ref{section_abstract_nu_metric}, we recall the
  general setup with the assumptions and the abstract metric $d_\nu$ from
  \cite{BalSas}.
\item In Section~\ref{section_Hinfty}, we specialize $R$ to a concrete
  ring of stable transfer functions, namely $R=H^\infty$, and show
  that our abstract assumptions hold in this particular case.
  Moreover in the Subsection~\ref{subsection_genuine_extension}, we will
  show that when our extended $\nu$-metric is restricted to rational
  plants, we obtain the classical $\nu$-metric, hence showing that we have
  obtained a genuine extension.
\end{enumerate}

\section{Recap of the abstract $\nu$-metric}
\label{section_abstract_nu_metric}

\noindent We recall the setup from \cite{BalSas}:
\begin{itemize}
\item[(A1)] $R$ is commutative integral domain with identity.
\item[(A2)] $S$ is a unital commutative complex semisimple Banach
  algebra with an involution $\cdot^*$, such that $R \subset S$.  We
  use $\inv S$ to denote the invertible elements of $S$.
\item[(A3)] There exists a map $\iota: \inv S \rightarrow G$, where
  $(G,+)$ is an Abelian group with identity denoted by $\circ$, and
  $\iota$ satisfies
\begin{itemize}
\item[(I1)] $\iota(ab)= \iota (a) +\iota(b)$ ($a,b \in \inv S$).
\item[(I2)] $\iota(a^*)=-\iota(a)$ ($a\in \inv S$).
\item[(I3)] $\iota$ is locally constant, that is, $\iota$ is continuous
  when $G$ is equipped with the discrete topology.
\end{itemize}
\item[(A4)] $x\in R \cap (\inv S)$ is invertible as an element of $R$
  if and only if $\iota(x)=\circ$.
\end{itemize}

\medskip

\noindent We recall the following standard definitions from the
factorization approach to control theory.

\subsection{The notation $\mF(R)$:}\label{subsec1} $\mF(R)$ denotes
the field of fractions of $R$.

\subsection{The notation $F^*$:} If $F\in R^{p\times m}$, then $F^*\in
S^{m\times p}$ is the matrix with the entry in the $i$th row and $j$th
column given by $F_{ji}^*$, for all $1\leq i\leq p$, and all $ 1\leq j
\leq m$.

\subsection{Right coprime/normalized coprime factorization:} For
 a matrix $P \in (\mF(R))^{p\times m}$, a factorization $P=ND^{-1}$,
where $N,D$ are matrices with entries from $R$, is called a {\em right
  coprime factorization of} $P$ if there exist matrices $X, Y$ with
entries from $R$ such that $ X N + Y D=I_m$.  If moreover
$ N^{*} N +D^{*} D =I_m$, then the right coprime factorization is
referred to as a {\em normalized} right coprime factorization of $P$.

\subsection{Left coprime/normalized coprime factorization:} For
 a matrix $P \in (\mF(R))^{p\times m}$, a factorization
$P=\widetilde{D}^{-1}\widetilde{N}$, where $\widetilde{N},\widetilde{D}$
are matrices with entries from $R$, is
called a {\em left coprime factorization of} $P$ if there exist
matrices $\widetilde{X}, \widetilde{Y}$ with entries from $R$ such
that $ \widetilde{N} \widetilde{X}+\widetilde{D} \widetilde{Y}=I_p.  $
If moreover $ \widetilde{N} \widetilde{N}^{*}
+\widetilde{D}\widetilde{D}^{*}=I_p, $ then the left coprime
factorization is referred to as a {\em normalized} left coprime
factorization of $P$.

\subsection{The notation $G, \widetilde{G}, K,\widetilde{K}$:}
\label{subsec5}
Given $P \in (\mF(R))^{p\times m}$ with normalized right and left
factorizations $P=N D^{-1}$ and $P= \widetilde{D}^{-1} \widetilde{N}$,
respectively, we introduce the following matrices with entries from
$R$:
$$
G=\left[ \begin{array}{cc} N \\ D \end{array} \right] \quad
\textrm{and} \quad \widetilde{G}=\left[ \begin{array}{cc}
    -\widetilde{D} & \widetilde{N} \end{array} \right] .
$$
Similarly, given a $C \in (\mF(R))^{m\times p}$ with normalized right
and left factorizations $C=N_C D_C^{-1}$ and $C= \widetilde{D}_C^{-1}
\widetilde{N}_C$, respectively, we introduce the following matrices
with entries from $R$:
$$
K=\left[ \begin{array}{cc} D_C \\ N_C \end{array} \right] \quad
\textrm{and} \quad \widetilde{K}=\left[ \begin{array}{cc}
    -\widetilde{N}_C & \widetilde{D}_C \end{array} \right] .
$$

\subsection{The notation $\mS(R,p, m)$:}\label{subsec6}
$\mS(R,p, m)$ denotes the set of all elements $P\in
(\mF(R))^{p\times m}$ that possess a normalized right coprime
factorization and a normalized left coprime factorization.

\medskip

We now recall the definition of the metric $d_\nu$ on $\mS(R, p, m)$.
But first we specify the norm we use for matrices with entries from
$S$.

\begin{definition}[$\|\cdot \|_{S,\infty}$]\label{def_sup_norm}
Let $\mathfrak{M}$ denote the maximal ideal space of the Banach
algebra $S$.  For a matrix $M \in S^{p\times m}$, we set
\begin{equation}
\label{norm}
\|M\|_{S,\infty}= \max_{\varphi \in \mathfrak{M}} \nm {\mathbf M}(\varphi) \nm,
\end{equation}
and refer to it as the {\em Gelfand norm}.
Here ${\mathbf M}$ denotes the entry-wise Gelfand transform of $M$,
and $\nm \cdot \nm$ denotes the induced operator norm from $\mC^{m}$
to $\mC^{p}$. For the sake of concreteness, we fix the standard
Euclidean norms on the vector spaces $\mC^{m}$ to $\mC^{p}$.
\end{definition}

The maximum in \eqref{norm} exists since $\mathfrak{M}$ is a compact
space when it is equipped with Gelfand topology, that is, the
weak-$\ast$ topology induced from $\calL( S; \mC)$. Since we have
assumed $S$ to be semisimple, the Gelfand transform
$$
\widehat{\;\cdot\;}:S \rightarrow\widehat{S}\; ( \subset
C(\mathfrak{M},\mC))
$$
is an injective map. If $M\in S^{1\times 1}=S$, then we note that there
are two norms available for $M$: the one as we have defined above,
namely $\|M\|_{S,\infty}$, and the norm $\|M\|_{S}$ of $M$ as an
element of the Banach algebra $S$. But throughout this article, we
will use the norm given by \eqref{norm}.

\begin{definition}[Abstract $\nu$-metric $d_\nu$]
\label{def_nu_metric}
For $P_1, P_2 \in \mS(R,p,m)$, with the normalized left/right
coprime factorizations
\begin{eqnarray*}
P_1&=& N_{1} D_{1}^{-1}= \widetilde{D}_{1}^{-1} \widetilde{N}_{1},\\
P_2&=& N_{2} D_{2}^{-1}= \widetilde{D}_{2}^{-1} \widetilde{N}_{2},
\end{eqnarray*}
we define
\begin{equation}
\label{eq_nu_metric}
d_{\nu} (P_1,P_2 ):=\left\{
\begin{array}{ll}
  \|\widetilde{G}_{2} G_{1}\|_{S,\infty} &
  \textrm{if } \det(G_1^* G_2) \in \inv S \textrm{ and }
  \\ 
 &\phantom{\textrm{if }\;}  \iota (\det (G_1^* G_2))=\circ, \\
  1 & \textrm{otherwise},
\end{array}\right.
\end{equation}
where the notation is as in Subsections~\ref{subsec1}-\ref{subsec6}.
\end{definition}

The following was proved in \cite{BalSas}:

\begin{theorem}
\label{thm_d_nu_is_a_metric}
$d_\nu$ given by \eqref{eq_nu_metric} is a metric on $\mS(R, p, m)$.
\end{theorem}

\begin{definition}
Given $P \in (\mF(R))^{p\times m}$ and $C\in (\mF(R))^{m\times p}$,
the {\em stability margin} of the pair $(P,C)$ is defined by
$$
\mu_{P,C}=
\left\{ \begin{array}{ll}
\|H(P,C)\|_{S,\infty}^{-1} &\textrm{if }P \textrm{ is stabilized by }C,\\
0 & \textrm{otherwise.}
\end{array}\right.
$$
\end{definition}

The number $\mu_{P,C}$ can be interpreted as a measure of the
performance of the closed loop system comprising $P$ and $C$: larger
values of $\mu_{P,C}$ correspond to better performance, with
$\mu_{P,C}>0$ if and only if $C$ stabilizes $P$.

The following was proved in \cite{BalSas}:

\begin{theorem}
  If $P_0,P\in \mS(R,p,m)$ and $C\in \mS(R, m,p)$, then $$\mu_{P,C}
  \geq \mu_{P_0,C}-d_{\nu}(P_0,P).$$
\end{theorem}

The above result says that stabilizability is a robust property of the
plant, since if $C$ stabilizes $P_0$ with a stability margin
$\mu_{P,C}>m$, and $P$ is another plant which is close to $P_0$ in the
sense that $d_\nu(P,P_0)\leq m$, then $C$ is also guaranteed to
stabilize $P$.

\section{The $\nu$-metric when $R=H^\infty$}
\label{section_Hinfty}

Let $H^\infty$ be the Hardy algebra, consisting of all bounded and
holomorphic functions defined on the open unit disk $\mD:= \{ z\in
\mC: |z| <1\}$.

We will now introduce a Banach algebra, $C_{\textrm b}({\mA_{\rho}})$,
which will serve as the Banach algebra $S$ in our abstract set up.

\begin{notation}
Given $\rho\in (0,1)$, let ${\mA_{\rho}}$ be the open annulus
$$
{\mA_{\rho}}:=\{z\in \mC: \rho<|z|<1\}.
$$
We set $
C_{\textrm b}({\mA_{\rho}})=\{F:{\mA_{\rho}}\rightarrow \mC: f \textrm{ is continuous and bounded on }{\mA_{\rho}}\}.$
\end{notation}

\begin{proposition}\label{prop_A2_part1}
Let $\rho \in (0,1)$. With the norm defined by
$$
\|F\|_\infty :=\sup_{z\in {\mA_{\rho}}}|F(z)| \textrm{ for }
F\in C_{\textrm{\em b}}({\mA_{\rho}}),
$$
$C_{\textrm{\em b}}({\mA_{\rho}})$ is a unital semisimple complex Banach
algebra with the involution $\cdot^\ast$ defined by 
$$
(F^\ast)(z)= \overline{F(z)}\quad (z\in {\mA_{\rho}},\;F\in C_{\textrm{\em b}}({\mA_{\rho}})).
$$
\end{proposition}
\begin{proof} The verification of the claims is straightforward.
We just give the proof of the semisimplicity. Recall that a
commutative complex Banach algebra is called semisimple if its radical
ideal, namely the intersection of all the maximal ideals of the Banach
algebra is $0$. We also know that kernels of complex homomorphisms are
maximal ideals. For $z\in {\mA_{\rho}}$, the map
$\varphi_{z}:C_{\textrm b}({\mA_{\rho}})\rightarrow \mC$, given by $
\varphi_{z}(F)= F(z)$ for $F\in C_{\textrm b}({\mA_{\rho}})$,
is a complex homomorphism. We have
$$
\bigcap_{z\in {\mA_{\rho}}} \ker \varphi_{z}=\{0\}.
$$
Since the radical ideal is contained in the intersection of the
kernels of the complex homomorphisms $\varphi_{z}$, $z\in {\mA_{\rho}}$,
it must be zero.
\end{proof}

\begin{proposition}\label{prop_A2_part2}
Let $\rho\in (0,1)$.  For
$f\in H^\infty$, define $\calI:H^\infty \rightarrow C_{\textrm{\em b}}({\mA_{\rho}})$ by
$$
(\calI (f))(z)=f(z) \quad (z\in {\mA_{\rho}},\; f \in H^\infty).
$$
Then $\calI$ is an injective map.
\end{proposition}
\begin{proof} The map $\calI$ is a linear transformation. Suppose that
  $\calI(f)=0$ for some $f\in H^\infty$. This means that the
  restriction of $f$ to the annulus ${\mA_{\rho}}$ is identically $0$, and as $f$ is holomorphic in $\mD$, $f$ must be zero in the whole disk $\mD$. Hence $f=0$.
\end{proof}

Henceforth we will identify $H^\infty$ as a subset of
$C_{\textrm b}({\mA_{\rho}})$ via this map $\calI$.

We will now define an index on invertible elements of $S=C_{\textrm b}({\mA_{\rho}})$.

\begin{notation}
We use the notation $C(\mT)$ for the Banach algebra of
complex-valued continuous functions defined on the unit circle
$\mT:=\{z\in \mC:|z|=1\}$, with all operations defined pointwise, 
with the supremum norm:
$$
\|f\|_\infty= \displaystyle \sup_{\zeta \in \mT}|f(\zeta)|\textrm{ for }f\in C(\mT),
$$
and the involution $\cdot^\ast$ defined pointwise:
$$
f^*(\zeta)=\overline{f(\zeta)}\quad (\zeta \in \mT).
$$
If $F\in \inv C_{\textrm b}({\mA_{\rho}})$, then for each
$r\in (\rho,1)$, the map $F_r:\mT \rightarrow \mC$, given by
$$
F_r(\zeta)=F(r\zeta) \quad (\zeta \in \mT),
$$
belongs to $ \inv C(\mT)$, and so each $F_r$ has a well-defined
(integral) winding number $w(F_r)\in \mZ$ with respect to $0$. 
\end{notation}

Moreover, we now show by the
local constancy of the winding number $w: \inv C(\mT) \rightarrow \mZ$, that
 $r\mapsto w(F_r)$ is constant on $(\rho,1)$.

\begin{proposition}
 If $F\in \textrm{\em inv } C_{\textrm{\em b}}({\mA_{\rho}})$, and $\rho<r<r'<1$, then
$$
w(F_r)=w(F_r').
$$
\end{proposition}
\begin{proof} We use the fact that the winding numbers $w(\varphi)$, $w(\psi)$ with respect
to $0$ of $\varphi,\psi:\mT\rightarrow  \mC\setminus \{0\}$, are the same if $\varphi$, $\psi$ are
homotopic; see for example \cite[\S2.7.10, p.50]{Die}.

As the annulus $
K:=\{z\in \mC: r\leq |z|\leq r'\}$
is compact, it follows that there is a $m>0$ such that $F(z)$
lies in $\mC\setminus m \mD$ for all $z\in K$. Also, $F$ is uniformly continuous on $K$, and so  we can choose 
$N$ large enough so that with 
$$
r_n:=r+(r'-r)\cdot \frac{n}{N}, \quad n=0,1,\dots, N,
$$
we have that 
$$
\|F_{r_n}-F_{r_{n+1}}\|_\infty<\frac{m}{2},\quad n=0,1,2,\dots, N-1 .
$$
Fix an $n\in \{0,1,2,\dots, N-1\}$. 
Set $\varphi=F_{r_n}$ and $\psi=F_{r_{n+1}}$. Then $\varphi$, $\psi$ belong to  $ \inv C(\mT)$.
Consider the map $H:\mT\times [0,1]\rightarrow \mC\setminus \{0\}$ defined by
$H(\zeta,t)=\varphi(\zeta)+t(\psi(\zeta)-\varphi(\zeta))$, $\zeta\in
\mT$, $t\in [0,1]$. Since
$$
|\varphi(\zeta)+t(\psi(\zeta)-\varphi(\zeta))|\geq
|\varphi(\zeta)|-|t(\psi(\zeta)-\varphi(\zeta))|\geq m-1\cdot
(m/2)=m/2>0,
$$
$H$ is well-defined. $H$ is a homotopy from $\varphi$ to
$\psi$. In particular it follows from the above that $\psi=H(\cdot,1)
\in \inv C(\mT)$, and that the winding numbers of $\varphi$ and $
\psi$ are identical. So it follows that 
$$
w(F_r)=w(F_{r_0})=w(F_{r_1})=\dots=w(F_{r_N})=w(F_r').
$$
This completes the proof.
\end{proof}

\begin{notation}
We now define the map $W:\inv C_{\textrm b}({\mA_{\rho}})\rightarrow \mZ$ by setting
$$
W(F)=w(F_r) \quad (r\in (\rho, 1), \;F\in \inv C_{\textrm b}({\mA_{\rho}})).
$$
\end{notation}

By the preceding discussion, we see that $W$ is well-defined.

We will now prove a sequence of results aimed towards verifying the assumptions (A3) and (A4)
in our abstract setup.

\begin{proposition}\label{proposition_I1}
Let $\rho \in (0,1)$. If $F,G\in \textrm{\em inv }C_{\textrm{\em b}}({\mA_{\rho}})$, then
$$
W(FG)=W(F)+W(G).
$$
\end{proposition}
\begin{proof} For $f,g\in \inv C(\mT)$, we have $w(fg)=w(f)+w(g)$,
and so it follows that for $F,G \in \inv C_{\textrm b}({\mA_{\rho}})$, and $r\in (\rho,1)$,
$$
W(FG)=w((FG)_r)=w(F_{r} \cdot G_{r})= w(F_{r})+w(G_{r})=  W(F)+W(G).
$$
This completes the proof.
\end{proof}

\begin{proposition}\label{proposition_I2}
Let $\rho \in (0,1)$. If $F\in \textrm{\em inv }C_{\textrm{\em b}}({\mA_{\rho}})$, then
$$
W(F^*)=-W(F).
$$
\end{proposition}
\begin{proof} For $f\in \inv C(\mT)$,
$w(\overline{f(\cdot)})=-w(f)$. So if $F \in
\inv C_{\textrm b}({\mA_{\rho}})$,
$$
W(F^*)=w((F^*)_{r})=w((F_{r})^*)=-w(F_{r})=-W(F).
$$
This completes the proof.
\end{proof}

\begin{proposition}\label{proposition_I3}
Let $\rho \in (0,1)$. Then $W:\textrm{\em inv } C_{\textrm{\em b}}({\mA_{\rho}})\rightarrow
\mZ$ is locally constant, that is, it is continuous when
$\mZ$ is equipped with the discrete topology.
\end{proposition}
\begin{proof} Let $F\in \inv C_{\textrm{b}}({\mA_{\rho}})$. Let $r\in (\rho,1)$.
By the local constancy of the map  $w: \inv C(\mT)\rightarrow \mZ$, it follows that there is a $\delta>0$
such that for all $h\in \inv C(\mT)$ satisfying $\|F_{r}-h\|_\infty<\delta$, we have $w(F_{r})=w(h)$.
Hence we have $W(F)=w(F_{r})=w(H_{r})=W(H)$
for all $H\in \inv C_{\textrm{b}}({\mA_{\rho}})$ satisfying $\|F-H\|_\infty<\delta$.
This proves the desired local constancy of $\overline{W}$.
\end{proof}

Finally we have the following analogue of the classical Nyquist
criterion.

\begin{proposition}\label{proposition_A4}
Let $\rho \in (0,1)$. Suppose that $f\in H^\infty$ is such that
$\calI(f)\in \textrm{\em inv }  C_{\textrm{\em b}}({\mA_{\rho}})$.
Then $f$ is invertible as an element of $H^\infty$ if and only
if $W(\calI(f))=0$.
\end{proposition}
\begin{proof} (``If" part) Let $g\in H^\infty$ be the inverse of $f$.
For each $r\in (\rho,1)$, $f_r\in A(\mD)$ defined by $
f_r(z)=f(r z)$ ($z\in \mD$), is invertible in $A(\mD)$. Then $(\calI(f))_r=f_r$.
By the Nyquist criterion for $A(\mD)$, $\varphi\in A(\mD)\bigcap \inv C(\mT)$ is invertible 
in $A(\mD)$ if and only if $w(\varphi)=0$
\cite[Lemma~5.2]{BalSas}. Thus $w(f_r)=0$. Hence $
W(\calI(f))=w((\calI(f))_r)=w(f_r)=0$, completing 
the proof of  the ``if'' part.

\medskip

\noindent (``Only if" part) Let $G \in C_{\textrm{b}}({\mA_{\rho}})$ be
the inverse of $F:=\calI(f)$. If $r\in (\rho,1)$,
then $f_r:=f(r\cdot)\in A(\mD)$ and $f_r\in \inv
C(\mT)$. Since $W(F)=w(f_r)=0$, it follows again by the Nyquist criterion for the
disk algebra recalled above, that $f_r$ is invertible in
$A(\mD)$. In other words, $f(r z)\neq 0$ for all $z\in \mD$. It follows from here,
as the choice of $r\in (\rho,1)$ was arbitrary, that $f(z)\neq 0$ for all $z\in \mD$, that is, $f$ has a pointwise inverse $g:\mD\rightarrow \mC$.
Moreover, $g$ is holomorphic in $\mD$.  We have
$f(z)g(z)=f(z) G(z)=1$ ($\rho<|z|<1$), and so it follows that
$G(z)=g(z)$ ($\rho<|z|<1$). Hence by the
maximum modulus principle,
$$
\sup_{z\in \mD} |g(z)|=\sup_{1>|z|>\rho} |g(z)|\leq \|G\|_\infty<+\infty,
$$
showing that $g\in H^\infty$. Consequently, $f\in \inv H^\infty$. This
completes the proof of the ``only if'' part.
\end{proof}

\begin{theorem}
\label{lemma_disk_algebra}
Let $\rho \in (0,1)$. Set
\begin{eqnarray*}
  R&:=& H^\infty, \\
  S&:=& C_{\textrm{\em b}}({\mA_{\rho}}),  \\
  G&:=& \mZ, \\
  \iota&:=& W.
\end{eqnarray*}
Then {\em (A1)-(A4)} are satisfied.
\end{theorem}
\begin{proof} Since $H^\infty$ is a commutative integral domain with
identity, (A1) holds.

(A2) follows from the results in Propositions~\ref{prop_A2_part1}
and \ref{prop_A2_part2}. Indeed, the set $C_{\textrm{b}}({\mA_{\rho}})$ is
a unital, commutative, complex, semisimple Banach algebra with the
involution $\cdot^\ast$ defined earlier in Proposition~\ref{prop_A2_part1}. Moreover, the map
$\calI:H^\infty\rightarrow C_{\textrm{b}}({\mA_{\rho}})$ is injective.

The map $W:\inv C_{\textrm{b}}({\mA_{\rho}})\rightarrow
\mZ$ satisfies (I1), (I2), (I3) by
Propositions~\ref{proposition_I1}, \ref{proposition_I2},
\ref{proposition_I3}. Thus (A3) holds.

Finally (A4) has been verified in Proposition~\ref{proposition_A4}.
\end{proof}

The definition of the abstract $\nu$-metric given in
Definition~\ref{def_nu_metric}, now takes the following concrete form.
For $P_1, P_2 \in \mS(H^\infty,p,m)$, with the normalized left/right
coprime factorizations
\begin{eqnarray*}
P_1&=& N_{1} D_{1}^{-1}= \widetilde{D}_{1}^{-1} \widetilde{N}_{1},\\
P_2&=& N_{2} D_{2}^{-1}= \widetilde{D}_{2}^{-1} \widetilde{N}_{2},
\end{eqnarray*}
we define
\begin{equation}
\label{eq_nu_metric_specialized}
d_{\nu} (P_1,P_2 )=\!\left\{\!
\begin{array}{ll}
  \|\widetilde{G}_{2} G_{1}\|_{C_{\textrm{b}}({\mA_{\rho}}),\infty} &
  \textrm{if }\det(G_1^* G_2) \in \inv C_{\textrm{b}}({\mA_{\rho}}) \textrm{ and }\\
  &\phantom{\textrm{if }\;}W (\det (G_1^* G_2))=0, \\
  1 & \textrm{otherwise},
\end{array} \right.
\end{equation}
where the notation is as in Subsections~\ref{subsec1}-\ref{subsec6}.

We will now show that in fact the Gelfand norm
$\|\cdot\|_{C_{\textrm{b}}({\mA_{\rho}}),\infty}$ above can be replaced by
the usual $\|\cdot\|_\infty$ norm for elements from $H^\infty$.

\begin{lemma} Let $\rho\in (0,1)$.
Let $A\in (H^\infty)^{p\times m}$. Then
$$
\|A\|_{C_{\textrm{\em b}}({\mA_{\rho}}),\infty}=\|A\|_\infty:= \sup_{z\in \mD} \nm A(z)\nm.
$$
\end{lemma}
\begin{proof} We first note that $C_{\textrm{b}}({\mA_{\rho}})$ is a
$C^*$-algebra. Indeed, for $F\in C_{\textrm{b}}({\mA_{\rho}})$,
$$
\|F^*F\|_\infty
= \sup_{z \in {\mA_{\rho}}} |\overline{F(z)}F(z)|= \sup_{z \in {\mA_{\rho}}} |F(z)|^2=\|F\|_\infty^2.
$$
Therefore (by the Gelfand-Naimark Theorem; see \cite[Theorem~11.18]{Rud}) for all $F\in  C_{\textrm{b}}({\mA_{\rho}})$, we
have
$$
\|F\|_\infty= \max_{\varphi
  \in \mathfrak{M}(C_{\textrm{b}}({\mA_{\rho}}))}
|\widehat{F}(\varphi)|=:\|F\|_{C_{\textrm{b}}({\mA_{\rho}}),\infty}.
$$
In the sequel, we use the notation
$\sigma_{\scriptstyle \textrm{max}}(X)$ (for $X\in
\mC^{p\times m}$) to mean the largest singular value of $X$, that is, the square
root of the largest eigenvalue of $X^*X$ or $XX^*$. The map $\sigma_{\scriptstyle \textrm{max}}(\cdot):
\mC^{p\times m}\rightarrow [0,\infty)$ is continuous.

Now let $F\in (C_{\textrm{b}}({\mA_{\rho}}) )^{p\times m}$.  Then $\sigma_{\scriptstyle
  \textrm{max}}(\widehat{F}(\cdot))$ is a continuous function on the maximal ideal space
$ \mathfrak{M}(C_{\textrm{b}}({\mA_{\rho}}))$, and so (by \cite[Theorem~11.18, p.289]{Rud}) there exists an element  $\mu_1 \in
C_{\textrm{b}}({\mA_{\rho}})$ such that
$$
\widehat{\mu_1}(\varphi)= \sigma_{\scriptstyle
  \textrm{max}}(\widehat{F}(\varphi)) \textrm{ for all }\varphi\in
\mathfrak{M}(C_{\textrm{b}}({\mA_{\rho}})).
$$
Also, the map $z \mapsto \sigma_{\scriptstyle
  \textrm{max}}(F(z))$ is continuous on ${\mA_{\rho}}$. 
 Moreover, we
have that
$$
\sup_{z\in {\mA_{\rho}}}\sigma_{\scriptstyle
  \textrm{max}}(F(z)) =\sup_{z\in {\mA_{\rho}}}
\nm F(z)\nm
<\infty.
$$
Consequently, if we define 
$
\mu_2(z):= \sigma_{\scriptstyle \textrm{max}}(F(z))$ ($z\in {\mA_{\rho}}$), 
then $\mu_2 \in C_{\textrm{b}}({\mA_{\rho}})$. 
This $\mu_2$ satisfies the equation 
 $
\det(\mu_2^2 I-A^\ast A)=0,$ 
which yields, by taking Gelfand transforms, that $
\det((\widehat{\mu_2}(\varphi))^2
I-(\widehat{A}(\varphi))^*\widehat{A}(\varphi))=0$ for all
$\varphi$ belonging to $ \mathfrak{M}(C_{\textrm{b}}({\mA_{\rho}}))$.
Hence there holds
\begin{equation}
\label{norms_are_same_1}
|\widehat{\mu_2}(\varphi)|\leq
\sigma_{\scriptstyle \textrm{max}}(\widehat{A}(\varphi))
=\widehat{\mu_1}(\varphi) \textrm{ for all }
\varphi\in \mathfrak{M}(C_{\textrm{b}}({\mA_{\rho}})).
\end{equation}
Also, since  $
\det((\widehat{\mu_1}(\varphi))^2
I-(\widehat{A}(\varphi))^*\widehat{A}(\varphi))=0$ for
   $\varphi\in \mathfrak{M}(C_{\textrm{b}}({\mA_{\rho}}))$,
it follows that $ \det(\mu_1^2 I-A^\ast A)=0$, which gives the
inequality
\begin{equation}
\label{norms_are_same_2}
|\mu_1(z)|\leq \sigma_{\scriptstyle \textrm{max}}(F(z))
=\mu_2(z)\textrm{ for all }
z\in {\mA_{\rho}}.
\end{equation}
It now follows from \eqref{norms_are_same_1} and
\eqref{norms_are_same_2} that $\|\mu_1\|_{C_{\textrm{b}}({\mA_{\rho}})}=
\|\mu_2\|_{C_{\textrm{b}}({\mA_{\rho}})}$, and so 
$$
\sup_{z\in {\mA_{\rho}}} \sigma_{\scriptstyle
  \textrm{max}}(F(z)) = \max_{\varphi\in \mathfrak{M}(C_{\textrm{b}}({\mA_{\rho}}))}\sigma_{\scriptstyle
  \textrm{max}}(\widehat{F}(\varphi)).
$$
Consequently, $
\|F\|_{C_{\textrm{b}}({\mA_{\rho}}),\infty}=\|F\|_\infty:= \displaystyle \sup_{z\in {\mA_{\rho}}}\nm F(z)\nm$.

Now suppose that $A \in (H^\infty)^{p\times m}$. Then we have
$$
\|A\|_{C_{\textrm{b}}({\mA_{\rho}}),\infty}= \sup_{z\in {\mA_{\rho}} }\nm A(z) \nm =\sup_{z\in \mD} \nm A(z) \nm
=\|A\|_\infty,
$$
we we have used the vector valued version of the Maximum Modulus
Principle (see for example \cite[p.50]{NikI}) to obtain the second
equality.  This completes the proof.
\end{proof}

In light of the above result, the abstract $\nu$-metric now takes the
following form. 

For $P_1, P_2 \in \mS(H^\infty,p,m)$, with the normalized left/right
coprime factorizations
\begin{eqnarray*}
P_1&=& N_{1} D_{1}^{-1}= \widetilde{D}_{1}^{-1} \widetilde{N}_{1},\\
P_2&=& N_{2} D_{2}^{-1}= \widetilde{D}_{2}^{-1} \widetilde{N}_{2},
\end{eqnarray*}
we define
\begin{equation}
\label{eq_nu_metric_specialized_norm}
d_{\nu} (P_1,P_2 ):=\left\{
\begin{array}{ll}
  \|\widetilde{G}_{2} G_{1}\|_{\infty} &
  \textrm{if } \det(G_1^* G_2) \in \inv C_{\textrm{b}}({\mA_{\rho}}) \textrm{ and }\\
  & \phantom{\textrm{if }\;}  W(\det (G_1^* G_2))=0, \\
  1 & \textrm{otherwise},
\end{array} \right.
\end{equation}
where the notation is as in Subsections~\ref{subsec1}-\ref{subsec6}.

\begin{remark}
\label{remark_stabilizable_is_S}
We also remark that the set $\mS(H^\infty,p,m)$ coincides with the set
of stabilizable plants, using the following two facts:
\begin{enumerate}
\item A plant is stabilizable over $H^\infty$ if and only if it
  possesses a coprime factorization. (See \cite{Ino} and \cite{Smi}.)
\item A {\em normalized} coprime factorization over $H^\infty$ exists
  whenever a coprime factorization exists over $H^\infty$. (See for
  example \cite[Theorem~1.1]{Mik}.)
\end{enumerate}
\end{remark}

Summarizing, our main result is the following, where  the
stability margin of a pair $(P,C)\in \mS(H^\infty,p,m)\times \mS(H^\infty, m,p)$ is
$$
\mu_{P,C}=
\left\{ \begin{array}{ll}
\|H(P,C)\|_{\infty}^{-1} &\textrm{if }P \textrm{ is stabilized by }C,\\
0 & \textrm{otherwise.}
\end{array}\right.
$$

\begin{corollary}
$d_\nu$ given by \eqref{eq_nu_metric_specialized_norm} is a metric
on the set of stabilizable plants over $H^\infty$. Moreover, if $P_0$,
$P$ belong to $ \mS(H^\infty,p,m)$ and $C\in \mS(H^\infty, m,p)$, then
$\mu_{P,C} \geq \mu_{P_0,C}-d_{\nu}(P_0,P)$.
\end{corollary}

\subsection{Irrelevance of $\rho\in (0,1)$
in the definition of the $\nu$-metric for $\mS(H^\infty,p,m)$}

Consider the condition
$$
(C): \fbox{$\det(G_1^* G_2) \in \inv C_{\textrm{b}}({\mA_{\rho}}) \textrm{ and }
W (\det (G_1^* G_2))=0.$}
$$
Clearly  only the tail end of the winding numbers are relevant, and
so the noninvertibility in $C_{\textrm{b}}({\mA_{\rho}})$ owing to the
noninveribility of $\det((G_1|_{r\mT})^*G_2|_{r\mT})$ for small $r$'s in $(\rho,1)$
should not really matter.  We remedy this
problem by taking the pointwise limit as $\rho\nearrow 1$ of the $\nu$-metrics
corresponding to the $\rho$'s in $(0,1)$.

For $\rho\in (0,1)$, let $d_\nu^\rho$ denote the $\nu$-metric given by
\eqref{eq_nu_metric_specialized_norm}. Define
$d_{\nu}^\infty$ on plant pairs from $\mS(H^\infty,p,m)$ as
follows. For $P_1,P_2\in \mS(H^\infty,p,m)$,
\begin{equation}
\label{eq_defn_d_nu_infty}
d_{\nu}^\infty(P_1,P_2):= \lim_{\rho\rightarrow 1} d_{\nu}^\rho (P_1,P_2).
\end{equation}
We note that if the condition (C) is satisfied corresponding to
$\rho$ for some $\rho\in (0,1)$, then it is also
satisfies for all $\rho'$ satisfying  $\rho\leq \rho'<1$.
This shows that the numbers $d_{\nu}^\rho(P_1,P_2)$, $\rho\in (0,1)$, are all equal
for all $\rho$'s  beyond a certain $\rho_{\textrm{c}}\in (0,1)$. Thus $d_{\nu}^\infty$,
given by \eqref{eq_defn_d_nu_infty}, is well-defined. We will now check that
$d_{\nu}^\infty$ is a metric on $\mS(H^\infty,p,m)$ and that with this metric,
stabilizability is a robust property of plants.

\begin{theorem}
$d_\nu^\infty$ given by \eqref{eq_defn_d_nu_infty} is a metric on
the set of stabilizable plants over $H^\infty$. Moreover, if $P_0,
P\in  \mS(H^\infty,p,m)$ and $C\in \mS(H^\infty, m,p)$, then $
\mu_{P,C} \geq \mu_{P_0,C}-d_{\nu}^\infty(P_0,P)$.
\end{theorem}
\begin{proof} We first show that $d_\nu^\infty$ defines a metric on
$\mS(H^\infty, p,m)$.
\begin{itemize}

\item[(D1)] For $P_1,P_2\in \mS(H^\infty, p,m)$, since $d_{\nu}^\rho(P_1,P_2)\geq 0$ for each $\rho\in (0,1)$,
$$
d_{\nu}^\infty(P_1,P_2)=\lim_{\rho\rightarrow 1} d_{\nu}^\rho (P_1,P_2)\geq 0.
$$
For $P\in \mS(H^\infty, p,m)$,
$
\displaystyle d_{\nu}^\infty(P,P)=\lim_{\rho\rightarrow 1} d_{\nu}^\rho
(P,P)=\lim_{\rho\rightarrow 1} 0=0.
$

\noindent Finally, if $d_{\nu}^\infty(P_1,P_2)= 0$ for $P_1,P_2\in \mS(H^\infty,
p,m)$, then since we have seen that the numbers $d_{\nu}^\rho(P_1,P_2)$, $\rho\in (0,1)$,
are all equal for all  $\rho$'s close enough to $1$,
it must be the case that $d_{\nu}^\rho(P_1,P_2)=0$ for all  $\rho$'s close enough to $1$,
and so $P_1=P_2$.

\item[(D2)] If $P_1,P_2\in \mS(H^\infty, p,m)$, then we have
$$
d_{\nu}^\infty(P_1,P_2)=\lim_{\rho\rightarrow 1} d_{\nu}^\rho (P_1,P_2)
=\lim_{\rho\rightarrow 1} d_{\nu}^\rho (P_2,P_1)=d_{\nu}^\infty(P_2,P_1).
$$

\item[(D3)] Finally, for all $P_1,P_2,P_3\in \mS(H^\infty, p,m)$,
passing the limit as $\rho\rightarrow 1$ in the
triangle inequalities
$$
d_{\nu}^\rho (P_1,P_3)\leq d_{\nu}^\rho (P_1,P_2)+d_{\nu}^\rho (P_2,P_3) \quad
(\rho \in (0,1)),
$$
yields the triangle inequality $
d_{\nu}^\infty (P_1,P_3)\leq d_{\nu}^\infty (P_1,P_2)+d_{\nu}^\infty
(P_2,P_3)$.
\end{itemize}
Thus $d_\nu^\infty$ defines a metric on $\mS(H^\infty, p,m)$. Next we
show that stabilizability is a robust property of the
plant. Let $P_0$, $P$ belong to $ \mS(H^\infty,p,m)$ and $C\in
\mS(H^\infty, m,p)$. Then $\mu_{P,C} \geq \mu_{P_0,C}-d_{\nu}^\rho(P_0,P)$ ($\rho\in (0,1)$).
Again passing the limit as $\rho\rightarrow 1$, we obtain $
\mu_{P,C} \geq \mu_{P_0,C}-d_{\nu}^\infty(P_0,P)$.
This completes the proof.
\end{proof}

\subsection{$d_\nu^\infty$ is an extension of the ``classical''
  $\nu$-metric}\label{subsection_genuine_extension}

In \cite{Vin}, the $\nu$-metric for rational plants (and more
generally elements of $\mS(A(\mD),p,m)$) was defined as follows.
For $P_1, P_2 \in \mS(A(\mD),p,m)$, with the normalized left/right
coprime factorizations
\begin{eqnarray*}
P_1&=& N_{1} D_{1}^{-1}= \widetilde{D}_{1}^{-1} \widetilde{N}_{1},\\
P_2&=& N_{2} D_{2}^{-1}= \widetilde{D}_{2}^{-1} \widetilde{N}_{2},
\end{eqnarray*}
we define
\begin{equation}
\label{eq_nu_metric_specialized_disk_algebra}
d_{\nu,{\scriptstyle\textrm{classical}}} (P_1,P_2 ):=\left\{
  \begin{array}{ll}
  \|\widetilde{G}_{2} G_{1}\|_{\infty} &
  \textrm{if } \det(G_1^* G_2) \in \inv C(\mT) \textrm{ and }\\
  & \phantom{\textrm{if }\;} w(\det(G_1^* G_2) )=0, \\
  1 & \textrm{otherwise},
\end{array} \right.
\end{equation}
where the notation is as in Subsections~\ref{subsec1}-\ref{subsec6}.

In this subsection we will show that our $\nu$-metric, defined by
\eqref{eq_defn_d_nu_infty}, coincides exactly with the above metric
defined by \eqref{eq_nu_metric_specialized_disk_algebra}, when the
data $P_1,P_2$ belong to $\mS(A(\mD),p,m)$ (instead of the bigger set
$\mS(H^\infty,p,m)$).

\begin{theorem}
Let $P_1, P_2 \in \mS(A(\mD),p,m)$. Then
$$
d_{\nu,{\scriptstyle\textrm{\em classical}}} (P_1,P_2)=d_{\nu}^\infty(P_1,P_2).
$$
\end{theorem}
\begin{proof} Let $d_{\nu,{\scriptstyle\textrm{classical}}}
(P_1,P_2 )<1$.  Then $\det(G_1^* G_2) \in \inv C(\mT)$. Since the
map $z\mapsto \det ( (G_1(z))^*G_2(z))$ is continuous on
$\overline{\mD}$, it follows that the
 two maps
$
\zeta\mapsto \det( (G_1(r \zeta))^* G_2(r \zeta))$  and $
\zeta \mapsto \det( (G_1( \zeta))^* G_2( \zeta))$ ($\zeta\in \mT$)
are close in the norm of $C(\mT)$ for all
$r$'s close enough to $1$. Consequently their winding numbers are equal. Hence it follows
that for a $\rho$ sufficiently close to $1$, when $\det(G_1^* G_2)$ is considered as an element $F$ of $C_{\textrm{b}}({\mA_{\rho}})$, it is invertible in 
$C_{\textrm{b}}({\mA_{\rho}})$, we have
that the $F_r$ are invertible in $C(\mT)$ for all $r$'s close enough to $1$, and their
winding numbers are $0$. Thus the condition (C) is satisfied for all $\rho$'s close enough to
$1$.  Hence $d_{\nu}^\rho(P_1,P_2)=\|\widetilde{G}_{2} G_{1}\|_{\infty}
=d_{\nu,{\scriptstyle\textrm{classical}}} (P_1,P_2
)$ for all $\rho$'s close enough to $1$. Consequently,
$d_{\nu}^\infty(P_1,P_2)=d_{\nu,{\scriptstyle\textrm{classical}}}
(P_1,P_2 )$ ($<1$).

\medskip

Now suppose that $d_{\nu}^\infty(P_1,P_2)<1$. Then
$d_{\nu}^\rho(P_1,P_2)$ is a constant $<1$ for all
$\rho$'s sufficiently close to $1$.
This implies that the condition (C) is satisfied for all $\rho$'s close enough to $1$.
Hence the maps
$$
\zeta\stackrel{\varphi_r}{\mapsto} \det( (G_1(r \zeta))^* G_2(r
\zeta))\quad (\zeta\in \mT)
$$
are all elements of $\inv C(\mT)$ for all $r$'s close enough to $1$, and
moreover, their winding numbers are all equal to $0$. Owing to the
invertibility in $C_{\textrm{b}}({\mA_{\rho}})$, it follows that these
maps $\varphi_r$ are uniformly bounded away from $0$ for all $r$'s
close enough to $1$. Also, these maps converge in $C(\mT)$ to the map
$$
\zeta \stackrel{\varphi}{\mapsto} \det( (G_1( \zeta))^* G_2(
\zeta))\quad (\zeta\in \mT).
$$
Hence $\varphi \in \inv C(\mT)$. Since the winding number map
$w:\inv C(\mT)\rightarrow \mZ$ is locally constant, we can also
conclude that $w(\varphi)=0$. Hence
$$
d_{\nu,{\scriptstyle\textrm{classical}}} (P_1,P_2 )
=\|\widetilde{G}_{2} G_{1}\|_{\infty}=d_\nu^\infty(P_1,P_2)\;(<1).
$$
This completes the proof.
\end{proof}

\subsection{$d_\nu^\infty$ is an extension of the
  $\nu$-metric defined for $R=QA$ in \cite{BalSas2}}
\label{subsection_nu_versus_old_for_QA}

In \cite{BalSas2}, a $\nu$-metric was defined when $R=QA$,
and we recall the definition below.

First of all, we use the notation $QC$ for the $C^*$-subalgebra
of $L^\infty(\mT)$ of {\em quasicontinuous} functions: $
QC:= (H^\infty + C(\mT) ) \bigcap \overline{(H^\infty + C(\mT) )}$.
The Banach algebra $QA$ of
analytic quasicontinuous functions is $QA:= H^\infty \bigcap QC$.
For $P_1, P_2 \in \mS(QA,p,m)$, with the normalized left/right
coprime factorizations
\begin{eqnarray*}
P_1&=& N_{1} D_{1}^{-1}= \widetilde{D}_{1}^{-1} \widetilde{N}_{1},\\
P_2&=& N_{2} D_{2}^{-1}= \widetilde{D}_{2}^{-1} \widetilde{N}_{2},
\end{eqnarray*}
we define
\begin{equation}
\label{eq_nu_metric_aaa}
d_{\nu} (P_1,P_2 ):=\left\{
\begin{array}{ll}
  \|\widetilde{G}_{2} G_{1}\|_{\infty} &
  \textrm{if } \det(G_1^* G_2) \in \inv QC \textrm{ and }\\
& \phantom{\textrm{if }}\; \textrm{Fredholm index of }
  T_{\det (G_1^* G_2)}=0, \\
  1 & \textrm{otherwise}. \end{array}
\right.
\end{equation}
where the notation is as in Subsections~\ref{subsec1}-\ref{subsec6}.

In this subsection we will show that our $\nu$-metric, defined by
\eqref{eq_defn_d_nu_infty}, coincides exactly with the above metric
defined by \eqref{eq_nu_metric_aaa}, when the
data $P_1,P_2$ belong to $\mS(QA,p,m)$ (instead of the bigger set
$\mS(H^\infty,p,m)$).

If $\varphi\in L^1(\mT)$, then  $\varphi_{(r)}$ ($0\leq r<1$) is the map defined by
$$
\varphi_{(r)}(\zeta)=(f\ast P_r)(\zeta)\quad  (\zeta\in \mT).
$$
Here $P_r$ denotes the Poisson kernel, given by
$$
P_r(\theta)=\sum_{k\in \mZ} r^{|k|} e^{ik\theta}\quad \theta \in[0,2\pi).
$$
Then it is straightforward to see that $(\varphi^\ast)_{(r)}=(\varphi_{(r)})^\ast$.
A result of Sarason \cite[Lemma~6]{Sar}
says that for $\varphi\in QC$ and $\psi\in L^\infty(\mT)$,
$$
\lim_{r\rightarrow 1} \|\varphi_{(r)} \psi_{(r)}-(\varphi \psi)_{(r)}\|_\infty =0.
$$
We will also use the result given below; see
 \cite[Theorem~7.36]{Dou}, \cite[Part~B, Theorem~4.5.10]{NikI}.

\begin{proposition}
\label{prop_Dou}
If $f\in H^\infty+C(\mT)$, then $T_f$ is Fredholm if and only if there exist
$\delta, \epsilon>0$ such that
$$
|f_{(r)}(e^{it})| \geq \epsilon \textrm{ for } 1-\delta <r<1, \;t\in [0,2\pi).
$$
Moreover, then the Fredholm index of $T_f$ (namely, $\dim( \ker T_f)-\dim( \ker T_f^\ast)$)
is the negative of the winding number with
respect to the origin of the curves $f_{(r)}$ for $1-\delta <r<1$.
\end{proposition}

\begin{theorem}
  Let $P_1, P_2 \in \mS(QA,p,m)$. Then
  $d_{\nu,QA} (P_1,P_2
  )=d_{\nu}^\infty(P_1,P_2)$.
\end{theorem}
\begin{proof} Let $d_{\nu,QA} (P_1,P_2 )<1$.  Then $\varphi:=\det(G_1^* G_2) \in \inv QC$.
But then it is also invertible as an element of $H^\infty+C(\mT)$. From Douglas's result
recalled above, we have that for all $r$ sufficiently close to $1$,
$\varphi_{(r)} \in \inv C(\mT)$, they are uniformly bounded away from $0$,
and their winding numbers are all equal to the Fredholm index of $T_\varphi$.

Using Sarason's result (\cite[Lemma~6]{Sar})
recalled above, and the local constancy of winding numbers,
we will now show that for all
$r$'s close enough to $1$ the maps $\zeta\stackrel{\varphi_r}{\mapsto} \det( (G_1(r \zeta))^* G_2(r \zeta))$
($\zeta\in \mT$) are invertible as elements of  $C(\mT)$, and moreover their winding numbers are all $0$.
Indeed, we have
$$
\det((G_1|_{r\mT})^\ast G_2|_{r\mT})=\sum_{i} (g_{1i}|_{r\mT})^* (g_{2i}|_{r\mT})
$$
for suitable scalar $g_{1i}, g_{2i} \in QA$ and indices $i$. But by \cite[Part~A, Section~3.4]{NikI},
$g_{1i}|_{r\mT}=g_{1i,(r)}$ and $g_{2i}|_{r\mT}=g_{2i,(r)}$.
Also, by Sarason's result, for all $i$'s
$$
\|g_{1i,(r)}^\ast g_{2i,(r)} -(g_{1i}^\ast g_{2i})_{(r)} \|_\infty
\stackrel{r\rightarrow 1}{\longrightarrow}
0.
$$
Hence
$
\| \varphi_r-\varphi_{(r)}\|_\infty  \stackrel{r\rightarrow 1}{\longrightarrow}0$.
Since for all
$r$'s close enough to $1$, the $\varphi_{(r)}$ are uniformly bounded away from $0$, it follows
that also the $\varphi_{r}$ are uniformly bounded away from $0$. In particular, they are all elements
of $\inv C(\mT)$ for $r$'s sufficiently near $1$. Finally, by the local constancy of winding numbers,
it follows that also the winding numbers of $\varphi_r$ are all $0$ for all $r$'s close enough to $1$.

Hence  when $\det(G_1^* G_2)$ is considered as an element $F$ of $C_{\textrm{b}}({\mA_{\rho}})$, we have
that $F_r$ are invertible in $C(\mT)$ for all $r$'s close enough to $1$, and their
winding numbers are $0$. Thus the condition (C) is satisfied for all $\rho$'s close enough to
$1$.  Hence $d_{\nu}^\rho(P_1,P_2)=\|\widetilde{G}_{2} G_{1}\|_{\infty}
=d_{\nu,QA} (P_1,P_2
)$ for all $\rho$'s close enough to $1$. Consequently,
$d_{\nu}^\infty(P_1,P_2)=d_{\nu,QA}
(P_1,P_2 )$ ($<1$).

\medskip

Now suppose that $d_{\nu}^\infty(P_1,P_2)<1$. Then
$d_{\nu}^\rho(P_1,P_2)$ is a constant $<1$ for all
$\rho$'s sufficiently close to $1$.
This implies that the condition (C) is satisfied for all $\rho$'s close enough to $1$.
Hence the maps
$$
\zeta\stackrel{\varphi_r}{\mapsto} \det( (G_1(r \zeta))^* G_2(r
\zeta))\quad (\zeta\in \mT)
$$
are all elements of $\inv C(\mT)$ for all $r$'s close enough to $1$, and
moreover, their winding numbers are all equal to $0$. Owing to the
invertibility in $C_{\textrm{b}}({\mA_{\rho}})$, it follows that these
maps $\varphi_r$ are uniformly bounded away from $0$ for all $r$'s
close enough to $1$. Set $\varphi$ to be the map
$$
\zeta\stackrel{\varphi}{\mapsto} \det( (G_1( \zeta))^* G_2(
\zeta))\quad (\zeta\in \mT)
$$
From the above observations, the maps $\varphi_{(r)}$ are uniformly bounded
away from $0$ for all $r$'s sufficiently near $1$ and moreover their winding numbers are all $0$.
But then by Douglas's result recalled above (or see \cite[Corollary~4.5.11]{NikI}),
the operator $T_\varphi$ is invertible. In particular, it is
Fredholm with Fredholm index $0$.  Hence $d_{\nu,QA} (P_1,P_2 )
=\|\widetilde{G}_{2} G_{1}\|_{\infty}=d_\nu^\infty(P_1,P_2)\;(<1)$.
This completes the proof.
\end{proof}

\section{A computational example}

Consider the transfer function $P$ given by
$$
P(s):= e^{-sT} \frac{s}{s-a},
$$
where $T,a>0$. Thus $P\in \mF(H^\infty(\mC_{{\scriptscriptstyle >0}}))$, where $H^\infty(\mC_{{\scriptscriptstyle >0}})$ denotes the set of bounded and holomorphic functions
defined in the open right half plane $\mC_{{\scriptscriptstyle >0}}:=\{s\in \mC: \textrm{Re}(s)>0\}$. With the conformal map
$
\varphi:\mD \rightarrow \mC_{{\scriptscriptstyle >0}}
$ given by
$$
\varphi(z)=\frac{1+z}{1-z} \quad (z\in \mD),
$$
we can then transplant the plant to the unit disk. In this manner, we can endow $\mS(H^\infty(\mC_{{\scriptscriptstyle >0}},p,m)$ also with the $\nu$-metric. 
 As an illustration, we will calculate the $\nu$-metric between a pair of plants
arising from this $P$ when there is uncertainty in the parameter $a$ or $T$. A normalized (left=right) coprime factorization of $P$ above is given by $P=N/D$,  where
$$
N(s)= \frac{se^{-s T}}{\sqrt{2} s+a},\quad
D(s)= \frac{s -a}{\sqrt{2} s +a}.
$$

\subsection{Uncertainty in $a$}

Consider the two plants
$$
P_1:=e^{-sT} \frac{s}{s-a_1} \textrm{ and } P_2:=e^{-sT} \frac{s}{s-a_2},
$$
where $T,a_1,a_2>0$. Set $s:=\varphi(z)$ for $z \in \mA_\rho$, $\rho\in (0,1)$. Then
$$
f(s):=G_1^* G_2=
 \frac{|s|^2 e^{-2 \textrm{Re}(s)T} +(\overline{s}-a_1)(s-a_2)}{(\sqrt{2}\overline{s}+a_1)(\sqrt{2}s+a_2)}
\quad (z \in \mA_\rho).
$$
It is clear that $z\mapsto |f(\varphi(z))|$ is bounded on $ \mD$.
It can be shown that  for $|a_1-a_2|$ small enough, the real part of $f(s)$ is
nonnegative and bounded away from zero for all $s \in \mC$ such that $\textrm{Re}(s)\geq 0$, as shown below.

\begin{lemma}
Let $T,a>0$. $\mC_{{\scriptscriptstyle \geq 0}}:= \{ s\in \mC: \textrm{\em Re}(s)\geq 0\}$ and set
$$
f(s):= \frac{|s|^2 e^{-2 \textrm{\em Re}(s)T} +(\overline{s}-a)(s-a-\delta)}{(\sqrt{2}\overline{s}+a)(\sqrt{2}s+a+\delta)} \quad (s\in \mC_{{\scriptscriptstyle \geq 0}}).
$$
Then there is a $\delta_0$ small enough such that for all $0\leq \delta<\delta_0$, 
there is a $m>0$ such that $\textrm{\em Re}(f(s))>m>0$ $(s\in \mC_{{\scriptscriptstyle \geq 0}})$.
\end{lemma}
\begin{proof}
Choose $\epsilon>0$ such that $
 \displaystyle\frac{2 \epsilon}{\sqrt{2}}+\epsilon^2<\frac{1}{4}$. We have 
$\displaystyle 
\lim_{{\scriptscriptstyle \begin{subarray}{l}|s|\rightarrow \infty\\ s\in \mC_{{\scriptscriptstyle \geq 0}}\end{subarray}}}
\frac{s-a}{\sqrt{2}s +a}=\frac{1}{\sqrt{2}}$. So we can choose a $R>0$ such that
$
\left| \displaystyle \frac{s-a}{\sqrt{2}s +a}-\frac{1}{\sqrt{2}}\right|<\displaystyle \frac{\epsilon}{2}.
$
We have
$$
\left|\frac{s-a}{\sqrt{2}s +a}-\frac{s-a-\delta}{\sqrt{2}s +a+\delta}\right|
=(1+\sqrt{2})|\delta| \frac{|s|}{|\sqrt{2} s+a|} \frac{1}{|\sqrt{2} s+a +\delta|}.
$$
It is easily seen that for all $s\in \mC_{{\scriptscriptstyle \geq 0}} $,
$
 \displaystyle\frac{|s|}{|\sqrt{2} s+a|}\leq \frac{1}{\sqrt{2}}$, 
and if $|s|\geq R$, then 
$$
\frac{1}{|\sqrt{2} s+a +\delta|}\leq \frac{1}{\sqrt{2} R}.
$$
So we have that $
\left| \displaystyle\frac{s-a}{\sqrt{2}s +a}-\frac{s-a-\delta}{\sqrt{2}s +a+\delta}\right|
\leq \displaystyle \frac{1+\sqrt{2}}{2 R}\cdot \delta.$ Choose $\delta_0$ so that
$$
\frac{1+\sqrt{2}}{2 R}\cdot \delta< \frac{\epsilon}{2}
$$
for all $0\leq \delta<\delta_0$.    
Thus whenever $|s|>R$, we have for all such $\delta$ that 
$$ 
\left|\frac{s-a-\delta}{\sqrt{2}s +a+\delta}-\frac{1}{\sqrt{2}} \right|<\epsilon.
$$
Hence $
\left| \displaystyle \frac{s-a}{\sqrt{2}s +a}\cdot \frac{s-a-\delta}{\sqrt{2}s +a+\delta}-\frac{1}{2} \right|<\displaystyle\frac{2 \epsilon}{\sqrt{2}}+\epsilon^2<\frac{1}{4}.
$
Thus 
$$
\frac{1}{2} -\textrm{Re}\left( \frac{s-a}{\sqrt{2}s +a}\cdot \frac{s-a-\delta}{\sqrt{2}s +a+\delta} \right)
\leq 
\left|\frac{s-a}{\sqrt{2}s +a}\cdot \frac{s-a-\delta}{\sqrt{2}s +a+\delta}-\frac{1}{2} \right|< \frac{1}{4},
$$
and so  $
\textrm{Re}\left( \displaystyle\frac{s-a}{\sqrt{2}s +a}\cdot \frac{s-a-\delta}{\sqrt{2}s +a+\delta} \right)>\displaystyle\frac{1}{4}.
$ But clearly for $s\in \mC_{{\scriptscriptstyle \geq 0}}$,
$$
\textrm{Re}\left(\frac{|s|^2 e^{-2 \textrm{Re}(s)T}}{(\sqrt{2}\overline{s}+a)(\sqrt{2}\overline{s}+a+\delta)}\right)
\geq 0.
$$
Hence $\textrm{Re}(f(s))\geq \displaystyle\frac{1}{4}$ for $|s|>R$ and $0\leq \delta<\delta_0$. 

Set $K:=\{s\in \mC: |s|\leq R\}\bigcap \mC_{{\scriptscriptstyle \geq 0}}$. Define $F:K\times [0,1]\rightarrow \mR$ by 
$$
F(s,\delta)=\textrm{Re}\left( \frac{|s|^2 e^{-2 \textrm{Re}(s)T} +(\overline{s}-a)(s-a-\delta)}{(\sqrt{2}\overline{s}+a)(\sqrt{2}s+a+\delta)} \right)
\quad  (s\in K, \;\delta\in [0,1]).
$$
Then $
F(s,0)=\textrm{Re}\left(\displaystyle \frac{|s|^2 e^{-2 \textrm{Re}(s)T}+|s-a|^2}{|\sqrt{2}\overline{s}+a|^2}\right)\geq 0.$
Set $
2m:=\displaystyle\min_{s\in K} F(s,0)$. 

\noindent Clearly $m\geq 0$. In fact, $m>0$ since if $F(s_0,0)=2m=0$ for some $s_0\in K$, then we would have 
$|s_0-a|^2=0$, and so $s_0=a$, but then 
$$
2m=|s|^2 e^{-2 \textrm{Re}(s)T}|_{s=s_0=a}\neq 0,
$$
a contradiction. As $F$ is continuous on the compact set $K\times [0,1]$, it is uniformly continuous there. 
Refine the choice of $\delta_0$ if necessary so that  $0\leq \delta<\delta_0$ implies that 
 $
|F(s,\delta)-F(s,0)|<m $ 
for all $s\in K$. Hence we have that $F(s,\delta)=\textrm{Re}(f(s))>m$ for all $0\leq \delta<\delta_0 $ and $s\in K$. 
This completes the proof.  
\end{proof}

In light of the above result, $G_1^* G_2 \in \inv C_{\textrm b}({\mA_{\rho}})$ for $\rho$ close enough to $1$, and $W(G_1^* G_2)=0$.
We also have
$$
\widetilde{G}_2 G_1=\frac{s e^{-sT}(a_2-a_1)}{(\sqrt{2}s+a_1)(\sqrt{2}s+a_2)},
$$
where $s:=\varphi(z)$, $z\in \mT\setminus \{1\}$.
Consequently, using the Cauchy-Schwarz (in)equality, we obtain
$$
\|\widetilde{G}_2 G_1\|_\infty=\frac{|a_2-a_1|}{2} \sup_{\omega\geq 0}
\frac{\omega}{\sqrt{\omega^2+\displaystyle\frac{a_1^2}{2}}\sqrt{\omega^2+\displaystyle\frac{a_2^2}{2}}}=
\frac{|a_2-a_1|}{2}\frac{\sqrt{2}}{a_1+a_2}.
$$
Hence
$$
d_\nu^\infty (P_1,P_2)= \frac{|a_1-a_2|}{\sqrt{2}(a_1+a_2)}
$$
whenever $|a_1-a_2|$ is small enough.

\subsection{Uncertainty in $T$}

Consider the two plants
$$
P_1:=e^{-sT_1} \frac{s}{s-a} \textrm{ and } P_2:=e^{-sT_2} \frac{s}{s-a},
$$
where $T_1,T_2,a>0$. We will show that $\|\widetilde{G}_2 G_1\|_\infty=1$, and 
so irrespective of whether or not the condition (C) is satisfied for some $\rho$, the $\nu$-metric between the plants will be always $1$. 

We have
$$
\widetilde{G}_2 G_1=\frac{s(s-1)(e^{-sT_2}-e^{-sT_1})}{2\left(s+\displaystyle \frac{1}{\sqrt{2}}\right)^2},
$$
where $s:=\varphi(z)$, $z\in \mT\setminus \{1\}$. Thus
$$
\|\widetilde{G}_2 G_1\|_\infty= \sup_{\omega\geq 0} \frac{\omega \sqrt{\omega^2+1}}{2 \sqrt{\omega^2+ \displaystyle \frac{1}{2}}}\sqrt{2} \sqrt{1-\cos (\omega(T_2-T_1))}.
$$
By the Arithmetic Mean-Geometric Mean inequality, we have for $\omega\geq 0$ that
$$
\omega^2 (\omega^2+1) \leq \left( \frac{\omega^2 +(\omega^2+1)}{2} \right)^2=\left(\omega^2+\frac{1}{2}\right)^2.
$$
We have
$$
\sup_{\omega\geq 0} \frac{\omega \sqrt{\omega^2+1}}{ \sqrt{\omega^2+ \displaystyle \frac{1}{2}}}=1=
\lim_{\omega \rightarrow \infty} \frac{\omega \sqrt{\omega^2+1}}{ \sqrt{\omega^2+ \displaystyle \frac{1}{2}}}.
$$
Also with
$$
\omega:= \frac{(2n+1) \pi}{T_2-T_1} \quad (n \in \mN)
$$
we have that $\omega \rightarrow \infty$, and $\cos (\omega (T_2-T_1))=-1$. Thus $
\|\widetilde{G}_2 G_1\|_\infty=1$, and so 
$$
d_{\nu}^\infty(P_1,P_2)=1.
$$

\medskip

\noindent {\bf Acknowledgements:} The author thanks
\begin{enumerate}
\item Reviewer 1 for creating a very  significant improvement
with the suggestion of  replacing an earlier choice of a Banach algebra $S$ by
 a smoother subalgebra of it, which greatly simplified the proofs;
\item  Joseph Ball for  the advice on comparing the $\nu$-metric
defined in this paper with the one from \cite{BalSas2},
and
for the suggestion of replacing an earlier sequential limit in
\eqref{eq_defn_d_nu_infty} with the continuous one now adopted;
\item  Kalle Mikkola for the
reference \cite{Mik} in Remark~\ref{remark_stabilizable_is_S};
\item Reviewer 2 for the suggestion of including an example illustrating the computation of the $\nu$-metric.
\end{enumerate}

\end{document}